\documentclass[journal]{IEEEtran}
\usepackage{amsthm}
\newtheorem{theorem}{Theorem}
\newtheorem{lemma}{Lemma}
\theoremstyle{definition}
\newtheorem{definition}{Definition}
\newtheorem{corollary}{Corollary}

\usepackage{algorithm}
\usepackage{algorithmic}

\ifCLASSINFOpdf
\else
   \usepackage[dvips]{graphicx}
   \graphicspath{{../eps/}}
   \DeclareGraphicsExtensions{.eps}
\fi
\ifCLASSOPTIONcompsoc
 \usepackage[caption=false,font=normalsize,labelfont=sf,textfont=sf]{subfig}
\else
 \usepackage[caption=false,font=footnotesize]{subfig}
\fi
\begin{document}
%
\title{Recovering Sparse Nonnegative Signals via Non-convex Fraction Function Penalty}
%
%
%

\author{Angang~Cui,
        Haiyang~Li,
        Meng~Wen,
        and~Jigen~Peng
\thanks{A. Cui and J. Peng are with the School of Mathematics and Statistics, Xi'an Jiaotong University, Xi'an, 710049, China.
e-mail: (cuiangang@163.com; jgpengxjtu@126.com).}
\thanks{H. Li and M. Wen are with the School of Science, Xi'an Polytechnic University, Xi'an, 710048, China. e-mail: (fplihaiyang@126.com; wen5495688@163.com).}
\thanks{Manuscript received, ; revised , .}}

%
%

\markboth{Journal of \LaTeX\ Class Files,~Vol.~, No.~, ~}%
{Shell \MakeLowercase{\textit{et al.}}: Bare Demo of IEEEtran.cls for IEEE Journals}
%



\maketitle

\begin{abstract}
Many real world practical problems can be formulated as $\ell_{0}$-minimization problems with nonnegativity constraints, which seek the sparsest nonnegative signals to underdetermined linear systems.
They have been widely applied in signal and image processing, machine learning, pattern recognition and computer vision. Unfortunately, this $\ell_{0}$-minimization problem with nonnegativity constraint
is computational and NP-hard because of the discrete and discontinuous nature of the $\ell_{0}$-norm. In this paper, we replace the $\ell_{0}$-norm with a non-convex
fraction function, and study the minimization problem of this non-convex fraction function in recovering the sparse nonnegative signals from an underdetermined linear system. Firstly, we discuss the
equivalence between $(P_{0}^{\geq})$ and $(FP_{a}^{\geq})$, and the equivalence between $(FP_{a}^{\geq})$ and $(FP_{a,\lambda}^{\geq})$. It is proved that the optimal solution of the problem $(P_{0}^{\geq})$
could be approximately obtained by solving the regularization problem $(FP_{a,\lambda}^{\geq})$ if some specific conditions satisfied. Secondly, we propose a nonnegative iterative thresholding algorithm to
solve the regularization problem $(FP_{a,\lambda}^{\geq})$ for all $a>0$. Finally, some numerical experiments on sparse nonnegative siganl recovery problems show that our method performs effective in
finding sparse nonnegative signals compared with the linear programming.
\end{abstract}


%
\IEEEpeerreviewmaketitle

\section{Introduction}\label{section1}
Many real world practical problems can be formulated as $\ell_{0}$-minimization problems with nonnegativity constraints, which seek the sparsest nonnegative signals to underdetermined linear systems.
They have been widely applied in signal and image processing (see, e.g., \cite{donoho9}, \cite{bar33}, \cite{bru34}, \cite{dono36}, \cite{kha37}, \cite{osg38}, \cite{moran39}, \cite{wang40}), machine
learning (see, e.g., \cite{bra41}, \cite{bra42}, \cite{he43}, \cite{mang44}, \cite{mang45}), pattern recognition and computer vision (see, e.g., \cite{moran39}, \cite{szlam46}), and so on. The
$\ell_{0}$-minimization problem with the nonnegativity constraint can be modeled into the following minimization
\begin{equation}\label{equ8}
(P_{0}^{\geq})\ \ \ \ \ \min_{x\in \Re^{n}}\|x\|_{0}\ \ \mathrm{subject}\ \mathrm{to}\ \ Ax=b,\ \ x\geq0
\end{equation}
where $A$ is a $m\times n$ real matrix of full row rank with $m\ll n$, $b$ is a nonzero real column vector of $m$-dimension, and $\|x\|_{0}$ is the so-called $\ell_0$-norm of real vector $x$, which counts the
number of the non-zero entries in $x$ (see, e.g., \cite{brucksteindm3}, \cite{elad4}, \cite{theodor5}). In general, the problem $(P_{0}^{\geq})$ is computational and NP-hard \cite{zhao49} because of the discrete
and discontinuous nature of the $\ell_{0}$-norm. A large amount of recent attention is attracted to the following minimization problem
\begin{equation}\label{equ9}
(P_{1}^{\geq})\ \ \ \ \ \min_{x\in \Re^{n}}\|x\|_{1}\ \ \mathrm{subject}\ \mathrm{to}\ \ Ax=b,\ \ x\geq0.
\end{equation}

The problem $(P_{1}^{\geq})$ has shown to be efficient for solving $(P_{0}^{\geq})$ in many situations (see, e.g., \cite{bruck47},\cite{dono48},\cite{zhao49},\cite{zhang50},\cite{kha51}), especially,
evidence in \cite{zhao49}, assuming the range space property (RSP) is adopted, the problem $(P_{1}^{\geq})$ can really make an exact recovery, and any linear programming solver can be used to
solve it. However, as the compact convex relaxation of the problem $(P_{0}^{\geq})$, the problem $(P_{1}^{\geq})$ may be suboptimal for recovering a real sparse signal.

Inspired by the good performances of the fraction function in image restoration and compressed sensing (see, e.g., \cite{geman31,li30}), in this paper, we replace the discontinuous $\ell_{0}$-norm $\|x\|_{0}$ with a continuous sparsity promoting penalty function
\begin{equation}\label{equ5}
P_{a}(x)=\sum_{i=1}^{n}\rho_{a}(x_{i}),\ \ \ a>0
\end{equation}
where
\begin{equation}\label{equ6}
\rho_{a}(t)=\frac{a|t|}{a|t|+1}
\end{equation}
is the fraction function, and it is increasing and concave in $t\in[0,+\infty]$.

\begin{figure}[h!]
 \centering
 \includegraphics[width=2.5in]{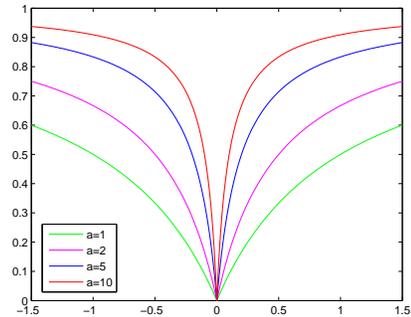}
\caption{The behavior of the fraction function $p_{a}(t)$ for various values of $a$.}
\label{fig:1}
\end{figure}

With the change of parameter $a>0$, the non-convex function $P_{a}(x)$ interpolates the $\ell_{0}$-norm
\begin{equation}\label{equ7}
\lim_{a\rightarrow+\infty}\rho_{a}(x_{i})=\left\{
    \begin{array}{ll}
      0, & {\ \ \mathrm{if} \ x_{i}=0;} \\
      1, & {\ \ \mathrm{if} \ x_{i}\neq 0.}
    \end{array}
  \right.
\end{equation}
Then, we translate problem $(P_{0}^{\geq})$ into the following minimization problem
\begin{equation}\label{equ10}
(FP_{a}^{\geq})\ \ \ \min_{x\in \Re^{n}}P_{a}(x)\ \ \mathrm{subject}\ \mathrm{to}\ \ Ax=b,\ \ x\geq0
\end{equation}
for the constrained form and
\begin{equation}\label{equ11}
(FP_{a,\lambda}^{\geq})\ \ \ \min_{x\geq0}\Big\{\|Ax-b\|_{2}^{2}+\lambda P_{a}(x)\Big\}
\end{equation}
for the regularization form.

The paper is organized as follows. In Section \ref{section2}, we establish the equivalences of $(P_{0}^{\geq})$, $(FP_{a}^{\geq})$ and $(FP_{a,\lambda}^{\geq})$. In section \ref{section3}, the nonnegative
iterative thresholding algorithm is proposed to solve the regularization problem $(FP_{a,\lambda}^{\geq})$ for all $a>0$. And the convergence of our algorithm is established in Section \ref{section4}. In
Section \ref{section5}, a series of experiments on some sparse nonnegative signal recovery problems are demonstrated. We conclude this paper in Section \ref{section6}.

\section{Equivalences of $(P_{0}^{\geq})$, $(FP_{a}^{\geq})$ and $(FP_{a,\lambda}^{\geq})$}\label{section2}

In this section, we first discus the equivalence between $(FP_{a}^{\geq})$ and $(P_{0}^{\geq})$, and then we study the equivalence between $(FP_{a,\lambda}^{\geq})$ and $(FP_{a}^{\geq})$.

\subsection{Equivalence between $(FP_{a}^{\geq})$ and $(P_{0}^{\geq})$}

Before our discussion, we give some notions and preliminary results that are used in later analysis.

\begin{definition}\label{de1}{\rm (\cite{luen53})}
Given the set of $m$ simultaneous linear equations in $n$ unknowns
\begin{equation}\label{equ12}
 Ax=b.
\end{equation}
Let $B$ be any nonsingular $m\times m$ sub-matrix made up of columns of $A$. Then, if all $n-m$ components of $x$ not associated with columns of $B$ are set equal to zero, the solution to the resulting
set of equations is said to be a basic solution to (\ref{equ12}) with respect to the basis $B$. The components of $x$ associated with columns of $B$ are called basic variables.
\end{definition}

\begin{definition}\label{de2}{\rm (\cite{chen54})}
Given a $m\times n$ matrix $A$ and a $m$-dimensional vector $b$, we define the linear problem is to find non-negative solution
$x\in \Re^{n}$ such that
\begin{equation}\label{equ13}
 Ax=b,\ \ x\geq0.
\end{equation}
We denote the problem by $\mathrm{LP}(A,b)$, its solution set by $\mathrm{SOL}(A,b)$ and its feasible set by $\mathrm{FEA}(A,b)=\{x|Ax=b, x\geq0\}$. A feasible solution to the
constraints (\ref{equ13}) that is also basic is said to be a basic feasible solution. The solution set $\mathrm{SOL}(A,b)$ often has an infinite
number of solutions when it is nonempty.
\end{definition}

\begin{definition}\label{de3}{\rm (\cite{luen53})}
If one or more of the basic variables in a basic solution has value zero, that solution is said to be a degenerate basic solution.
\end{definition}
\begin{definition}\label{de4}{\rm (\cite{luen53})}
A point $x$ in a convex set $\mathcal{C}$ is said to be an extreme point of $\mathcal{C}$ if there are no two distinct points $x_{1}$ and $x_{2}$ in $\mathcal{C}$ such that $x=\eta x_{1}+(1-\eta)x_{2}$ for some $\eta$, $0<\eta<1$.
\end{definition}

An extreme point is thus a point that does not lie strictly within a line segment connecting two other points of the set. The extreme
points of a triangle, for example, are its three vertices (see, e.g., \cite{luen53}).

\begin{theorem}\label{th1}{\rm (\cite{luen53})}
(Equivalence of extreme points and basic solutions) Let $A$ be an $m\times n$ matrix of rank $m$ and $b$ an $m$-dimension vector.
Let $\mathcal{D}$ be the convex polytope consisting of all $n$-dimension vectors $x$ satisfying (\ref{equ13}). Then, a vector $x$
is an extreme point of $\mathcal{D}$ if and only if $x$ is a basic feasible solution to (\ref{equ13}).
\end{theorem}

\begin{corollary}\label{co1}{\rm (\cite{luen53})}
If the convex set $\mathcal{D}$ corresponding to (\ref{equ13}) is nonempty, it has at least one extreme point.
\end{corollary}
\begin{corollary}\label{co2}{\rm (\cite{luen53})}
If there is a finite optimal solution to a linear programming problem, there is a finite optimal solution which is an extreme point of the constraint set.
\end{corollary}
\begin{corollary}\label{co3}{\rm (\cite{luen53})}
The constraint set $\mathcal{D}$ corresponding to (\ref{equ13}) possesses at most a finite number of extreme points.
\end{corollary}

Equipped above preliminary results, we shall establish the equivalence of the problems $(FP_{a}^{\geq})$ and $(P_{0}^{\geq})$.

By Definition \ref{de2}, the problems $(P_{0}^{\geq,1})$ and $(FP_{a}^{\geq,1})$ could be rewritten as
\begin{equation}\label{equ14}
(SOLP_{0}^{\geq})\ \ \ \min_{x\in \Re^{n}}\|x\|_{0}\ \ \mathrm{subject}\ \mathrm{to}\ \ x\in \mathrm{SOL}(A,b)
\end{equation}
and
\begin{equation}\label{equ15}
(SOLFP_{a}^{\geq})\ \ \ \min_{x\in \Re^{n}}P_{a}(x)\ \ \mathrm{subject}\ \mathrm{to}\ \ x\in \mathrm{SOL}(A,b).
\end{equation}
In particular, we call a solution of $(SOLFP_{a}^{\geq})$ a least fraction solution.

\textbf{Full rank assumption}: The $m\times n$ matrix $A$ has $m<n$, and the $m$ rows of $A$ are linearly independent. Otherwise, we make row
transformations simultaneously in both sides of the equation $Ax=b$, resulting in an equivalent equation $A_{1}x =\hat{b}$ with $A_{1}$ being
of full row rank.

\begin{lemma}\label{le1}
All least fraction solutions of the $\mathrm{LP}(A,b)$ are extreme points of $\mathrm{SOL}(A,b)$.
\end{lemma}

\begin{proof}
Let $x^{\ast}$ be a least fraction solution. Suppose there exist $y, z\in \mathrm{SOL}(A,b)$ such that
$x^{\ast}=\eta y+(1-\eta)z$ for some $0<\eta<1$. Recall that $\rho_{a}(t)$ is strictly concave for $t\geq0$. Then it follows
\begin{eqnarray*}
P_{a}(x^{\ast})&=&\sum_{i=1}^{n}\rho_{a}(x_{i}^{\ast})\\
&=&\sum_{i=1}^{n}\rho_{a}(\eta y_{i}+(1-\eta)z_{i})\\
&\geq&\eta\sum_{i=1}^{n}\rho_{a}(y_{i})+(1-\eta)\sum_{i=1}^{n}\rho_{a}(z_{i})\\
&=&\eta P_{a}(y_{i})+(1-\eta)P_{a}(z_{i})\\
&\geq&P_{a}(x^{\ast})
\end{eqnarray*}
where the last inequality uses that $x^{\ast}$ is a least fraction solution. Furthermore, the above equalities hold if and only if
$y=z=x^{\ast}$, which indicates that $x^{\ast}$ is an extreme point of $\mathrm{SOL}(A,b)$.
\end{proof}

By Lemma \ref{le1}, $x^{*}$ is a extreme point of the polytope set $\mathcal{D}$. We denote by $E(\mathcal{D})$ the set of extreme points of the polytope
set $\mathcal{D}$, and define two constants $r(A,b)$ and $R(A,b)$ as follows

\begin{equation}\label{equ16}
r(A,b)=\min_{z\in E(\mathcal{D}),z_i>0, 1\leq i\leq n}z_{i}.
\end{equation}
\begin{equation}\label{equ17}
R(A,b)=\max_{z\in E(\mathcal{D}),z_i>0, 1\leq i\leq n}z_{i}.
\end{equation}
Clearly, the defined constant $r(A,b)$ and $R(A,b)$ are finite and positive due to the finiteness of $E(\mathcal{D})$ and positive of $z_{i}$.\\

\begin{theorem}\label{th2}
There exists some constants $\hat{a}>0$ such that the optimal solution to the problem $(FP_{\hat{a}}^{\geq})$ also solves $(P_{0}^{\geq})$.
\end{theorem}

\begin{proof}
Let $\{a_{i}|i=0, 1,2,\cdots\}$ be a increasing infinite sequence with $\lim_{i\rightarrow\infty}a_{i}=\infty$ and $a_{0}=1$. For each $a_{i}$, by Lemma \ref{le1}, the optimal solution $\hat{x}_{i}$ to
the problem $(SOLFP_{a_{i}}^{\geq})$ is an extreme point of the polytope set $\mathcal{D}$. Since the polytope set $\mathcal{D}$ has a finite number of extreme points (see Theorem \ref{th1},
Corollary \ref{co1}, \ref{co2}, \ref{co3}), one extreme point, named $\hat{x}$, will repeatedly solves the problem $(SOLFP_{a_{i}}^{\geq})$ for some subsequence $\{a_{i_{k}}\mid k=1,2,\cdots\}$ of $\{a_i\}$. For
any $a_{i_{k}}\geq a_{i_{1}}$ and $x\in\mathcal{R}^n$, we have
$$P_{a_{i_{k}}}(x_{i_{k}})=\min P_{a_{i_{k}}}(x)\leq\|x\|_{0}.$$
Letting $i_{k}\rightarrow\infty$, we have
$$\|\hat{x}\|_{0}\leq \|x\|_{0}.$$
Hence $\hat{x}$ is the optimal solution to the problem $(SOLP_{0}^{\geq})$. This proves that there exists some constant $\hat{a}>0$ such that the optimal solution to
the problem $(FP_{\hat{a}}^{\geq})$ also solves $(P_{0}^{\geq})$.
\end{proof}

Furthermore, we have:
\begin{theorem}\label{th3}
There exists a constant $a^{*}>0$ such that, whenever $a>a^{*}$, every optimal solution to the problem $(FP_{a}^{\geq})$ also solves
$(P_{0}^{\geq})$, where $a^{*}$ depends on $A$ and $b$.
\end{theorem}

\begin{proof}
Let $x^{*}$ be the optimal solution to the problem $(SOLFP_{a}^{\geq})$ and $x^{0}$ be the optimal solution to the problem $(SOLP_{0}^{\geq})$.
By Lemma \ref{le1} we know that $x^{*}$ is an extreme point of the polytope set $\mathcal{D}$.

Therefore, we have
\begin{eqnarray*}
\min_{x\in \mathrm{SOL}(A,b)}\|x\|_{0}&=&\|x^{0}\|_{0}\\
&\geq&\sum_{i\in \mathrm{supp}(x^{0})}\frac{a|x_{i}^{0}|}{1+a|x_{i}^{0}|}\\
&\geq&\sum_{i\in \mathrm{supp}(x^{\ast})}\frac{a|x_{i}^{\ast}|}{1+a|x_{i}^{\ast}|}\\
&\geq&\|x^{\ast}\|_{0}\frac{a|x_{i}^{\ast}|}{1+a|x_{i}^{\ast}|}
\end{eqnarray*}
which implies that
$$\|x^{*}\|_{0}\leq (1+\frac{1}{ar})\|x^{0}\|_{0}=(1+\frac{1}{ar})\min_{x\in \mathrm{SOL}(A,b)}\|x\|_{0}.$$
Because
$\|x^{*}\|_{0}$ is an integer number, from the inequality above, it follows that  $\|x^{*}\|_{0}=\displaystyle
\min_{x\in \mathrm{SOL}(A,b)}\|x\|_{0}$ (that is, $x^{*}$ solves $(SOLP_{0}^{\geq})$) when
\begin{equation}\label{equ18}
(1+\frac{1}{ar(A,b)})\min_{x\in \mathrm{SOL}(A,b)}\|x\|_{0}<\min_{x\in \mathrm{SOL}(A,b)}\|x\|_{0}+1
\end{equation}
Obviously, the inequality (\ref{equ18}) is true whenever
\begin{equation}\label{equ19}
a>\frac{\displaystyle\min_{x\in \mathrm{SOL}(A,b)}\|x\|_{0}}{r(A,b)}.
\end{equation}
Therefore, with $a^{*}$ denoting the right side of the inequality (\ref{equ19}), we conclude that when $a>a^{*}$, every solution $x^{*}$ to
the problem $(SOLFP_{a}^{\geq})$ also solves $(SOLP_{0}^{\geq})$. This proves that whenever $a>a^{*}$, every solution to the problem
$(FP_{a}^{\geq})$ also solves $(P_{0}^{\geq})$.
\end{proof}

\subsection{Equivalence between $(FP_{a,\lambda}^{\geq})$ and $(FP_{a}^{\geq})$}

In this subsection, we study the equivalence of the regularization problem $(FP_{a,\lambda}^{\geq})$ and the constrained problem $(FP_{a}^{\geq})$.

\begin{theorem}\label{th4}
Let $\{\lambda_{\tilde{n}}\}$ be a decreasing sequence of
positive numbers with $\lambda_{\tilde{n}}\rightarrow 0$, and $x_{\lambda_{\tilde{n}}}$ be a global minimizer of the problem $(FP_{a,\lambda}^{\geq})$ with
$\lambda=\lambda_{\tilde{n}}$. If the problem $(FP_{a}^{\geq})$  is feasible, then the sequence $\{x_{\lambda_{\tilde{n}}}\}$ is bounded and any of
its accumulation points is a global minimizer of the problem $(FP_{a}^{\geq})$.
\end{theorem}

\begin{proof}
By
$$\lambda_{\tilde{n}} P_{a}(x)\leq\|Ax-b\|_{2}^{2}+\lambda_{\tilde{n}} P_{a}(x),$$
we can see that the objective function in the problem $(FP_{a,\lambda}^{\geq})$ with $\lambda=\lambda_{\tilde{n}}$
is bounded from below and is coercive, i.e.,
$$\|Ax-b\|_{2}^{2}+\lambda_{\tilde{n}} P_{a}(x)\rightarrow+\infty\ \ \mathrm{as}\ \ \|x\|_{2}\rightarrow+\infty,$$
and hence the set of global minimizers of $(FP_{a,\lambda}^{\geq})$ with $\lambda=\lambda_{\tilde{n}}$ is nonempty and bounded.

By assumption, we suppose that the problem $(FP_{a}^{\geq})$ is feasible and $\bar{x}$ is any feasible point, then
$$A\bar{x}=b.$$
Since $\{x_{\lambda_{\tilde{n}}}\}$ is a global minimizer of the problem $(FP_{a,\lambda}^{\geq})$ with $\lambda=\lambda_{\tilde{n}}$, we have

\begin{equation}\label{equ20}
\begin{array}{llll}
\lambda_{\tilde{n}}P_{a}(x_{\lambda_{\tilde{n}}})&\leq&\|Ax_{\lambda_{\tilde{n}}}-b\|_{2}^{2}+\lambda_{\tilde{n}}P_{a}(x_{\lambda_{\tilde{n}}})\\
&\leq&\|A\bar{x}-b\|_{2}^{2}+\lambda_{\tilde{n}} P_{a}(\bar{x})\\
&=&\lambda_{\tilde{n}} P_{a}(\bar{x}).
\end{array}
\end{equation}
Hence, the sequence $\{P_{a}(x_{\lambda_{\tilde{n}}})\}_{\tilde{n}\in N^{+}}$ is bounded, and the sequence $\{x_{\lambda_{\tilde{n}}}\}$ has at least one
accumulation point. In addition, by inequality (\ref{equ20}), we can get that
$$\|Ax_{\lambda_{\tilde{n}}}-b\|_{2}^{2}\leq\lambda_{\tilde{n}} P_{a}(\bar{x})\ \ \mathrm{for}\ \mathrm{any}\ \ \lambda_{\tilde{n}}\rightarrow 0.$$
If we set $x^{\ast}$ be any accumulation point of the sequence $\{x_{\lambda_{\tilde{n}}}\}$, we can derive that
$$Ax^{\ast}=b.$$
That is, $x^{\ast}$ is a feasible point of the problem $(FP_{a}^{\geq})$. Combined with $P_{a}(x^{\ast})\leq P_{a}(\bar{x})$ and the arbitrariness of $\bar{x}$,
we can get that $x^{\ast}$ is a global minimizer of $(FP_{a}^{\geq})$.
\end{proof}

Theorem \ref{th2} and \ref{th3} demonstrate that the optimal solution to the problem $(P_{0}^{\geq})$ can be exactly obtained by solving $(FP_{a}^{\geq})$ if some specific
conditions satisfied. Theorem \ref{th4} displays that the optimal solution to the problem $(FP_{a}^{\geq})$ can be approximately obtained by solving $(FP_{a,\lambda}^{\geq})$ for
some proper smaller $\lambda>0$.

\section{Nonnegative iterative thresholding (NIT) algorithm  for solving $(FP_{a,\lambda}^{\geq})$}\label{alg}\label{section3}

In this section, the nonnegative iterative thresholding (NIT) algorithm is proposed to solve the problem $(FP_{a,\lambda}^{\geq})$
for all $a\geq0$. Before we introduce the NIT algorithm, there are some results need to be prepared.

\subsection{Export the NIT algorithm}

\begin{lemma}\label{le2}
Define three threshold values
$$t_{1}^{\ast}=\frac{\sqrt[3]{\frac{27}{8}\lambda a^{2}}-1}{a}, \ \ \ \ t_{2}^{\ast}=\frac{\lambda}{2}a, \ \ \ \ t_{3}^{\ast}=\sqrt{\lambda}-\frac{1}{2a}$$
for any positive parameters $\lambda$ and $a$, then the inequalities $t_{1}^{\ast}\leq t_{3}^{\ast}\leq t_{2}^{\ast}$ hold. Furthermore,
they are equal to $\frac{1}{2a}$ when $\lambda=\frac{1}{a^{2}}$.
\end{lemma}

\begin{theorem}\label{th5}
Given any vector $v\in \Re^{n}$, the thresholding operator $\mathcal{T}_{a,\lambda}: \Re^{n}\rightarrow \Re^{n}$ defined by
$$\mathcal{T}_{a,\lambda}(v):=\arg\min_{x\in \Re^{n}}\Big\{\|x-v\|_{2}^{2}+\lambda P_{a}(x)\Big\},$$
can be expressed as
\begin{equation}\label{equ21}
\mathcal{T}_{a,\lambda}(v_{i})=\left\{
    \begin{array}{ll}
      g_{a,\lambda}(v_{i}), & \ \ \mathrm{if} \ {|v_{i}|> t^{\ast};} \\
      0, & \ \ \mathrm{if} \ {|v_{i}|\leq t^{\ast}.}
    \end{array}
  \right.
\end{equation}
where $g_{a,\lambda}(v_{i})$ is defined as
\begin{equation}\label{equ22}
g_{a,\lambda}(v_{i})=\mathrm{sgn}(v_{i})(\frac{\frac{1+a|v_{i}|}{3}(1+2\cos(\frac{\phi(|v_{i}|)}{3}-\frac{\pi}{3}))-1}{a}),
\end{equation}
$$\phi(t)=\arccos(\frac{27\lambda a^{2}}{4(1+a|t|)^{3}}-1)$$
and the threshold function satisfies
\begin{equation}\label{equ23}
t^{\ast}=\left\{
    \begin{array}{ll}
     t_{2}^{\ast}=\frac{\lambda}{2}a, & \ \ \mathrm{if} \ {\lambda\leq \frac{1}{a^{2}};} \\
     t_{3}^{\ast}=\sqrt{\lambda}-\frac{1}{2a}, & \ \ \mathrm{if} \ {\lambda>\frac{1}{a^{2}}.}
    \end{array}
  \right.
\end{equation}
\end{theorem}

The proof of Theorem \ref{th5} used the Cartan¡¯s root-finding formula expressed in terms of hyperbolic functions and it is  a special case of the reference \cite{xing27}, and the detailed proof can be seen in \cite{li30}.

\begin{figure}[h!]
  \begin{minipage}[t]{0.49\linewidth}
  \centering
  \includegraphics[width=1.1\textwidth]{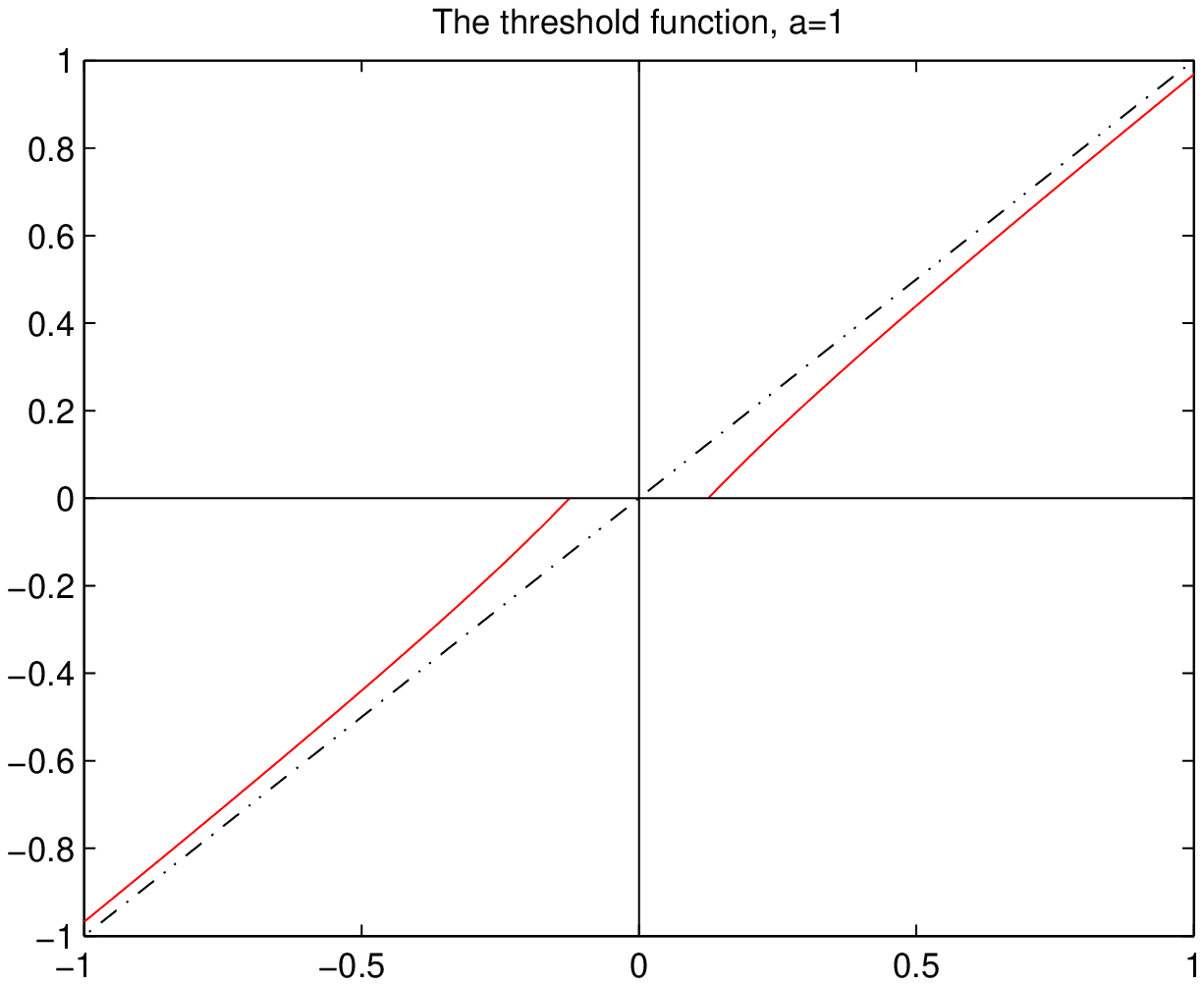}
  \end{minipage}
  \begin{minipage}[t]{0.49\linewidth}
  \centering
  \includegraphics[width=1.1\textwidth]{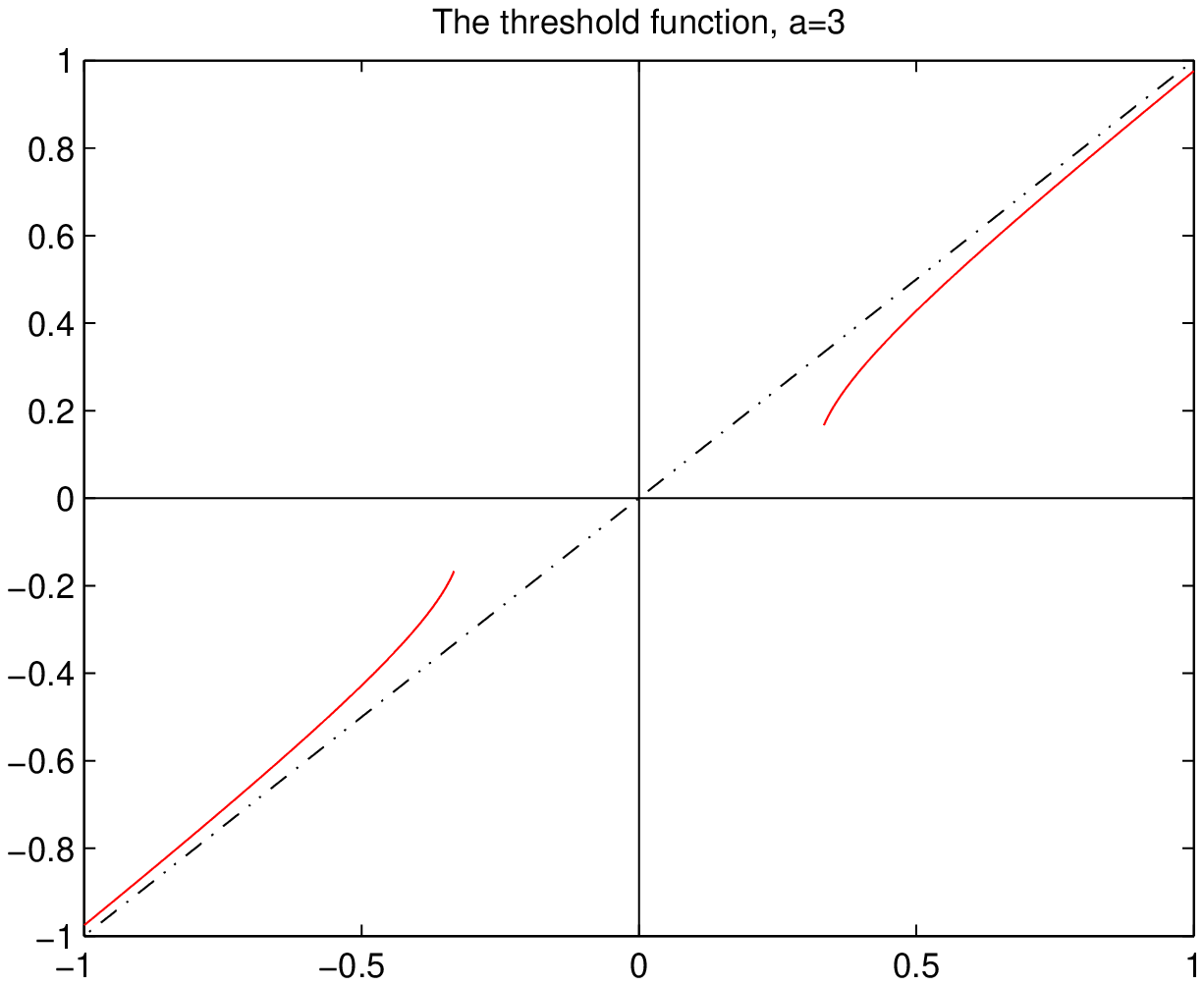}
  \end{minipage}
  \begin{minipage}[t]{0.49\linewidth}
  \centering
  \includegraphics[width=1.1\textwidth]{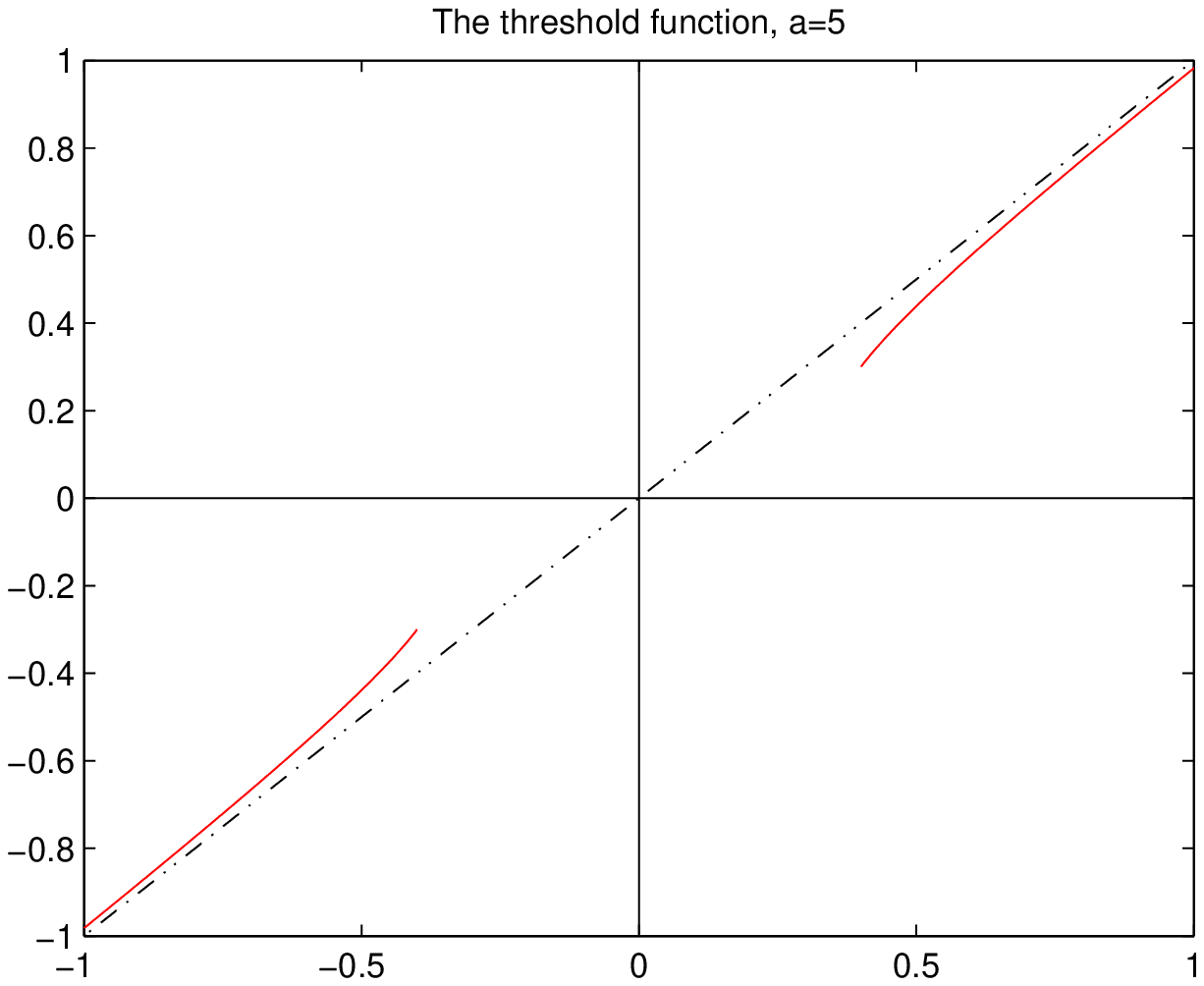}
  \end{minipage}
  \begin{minipage}[t]{0.49\linewidth}
  \centering
  \includegraphics[width=1.1\textwidth]{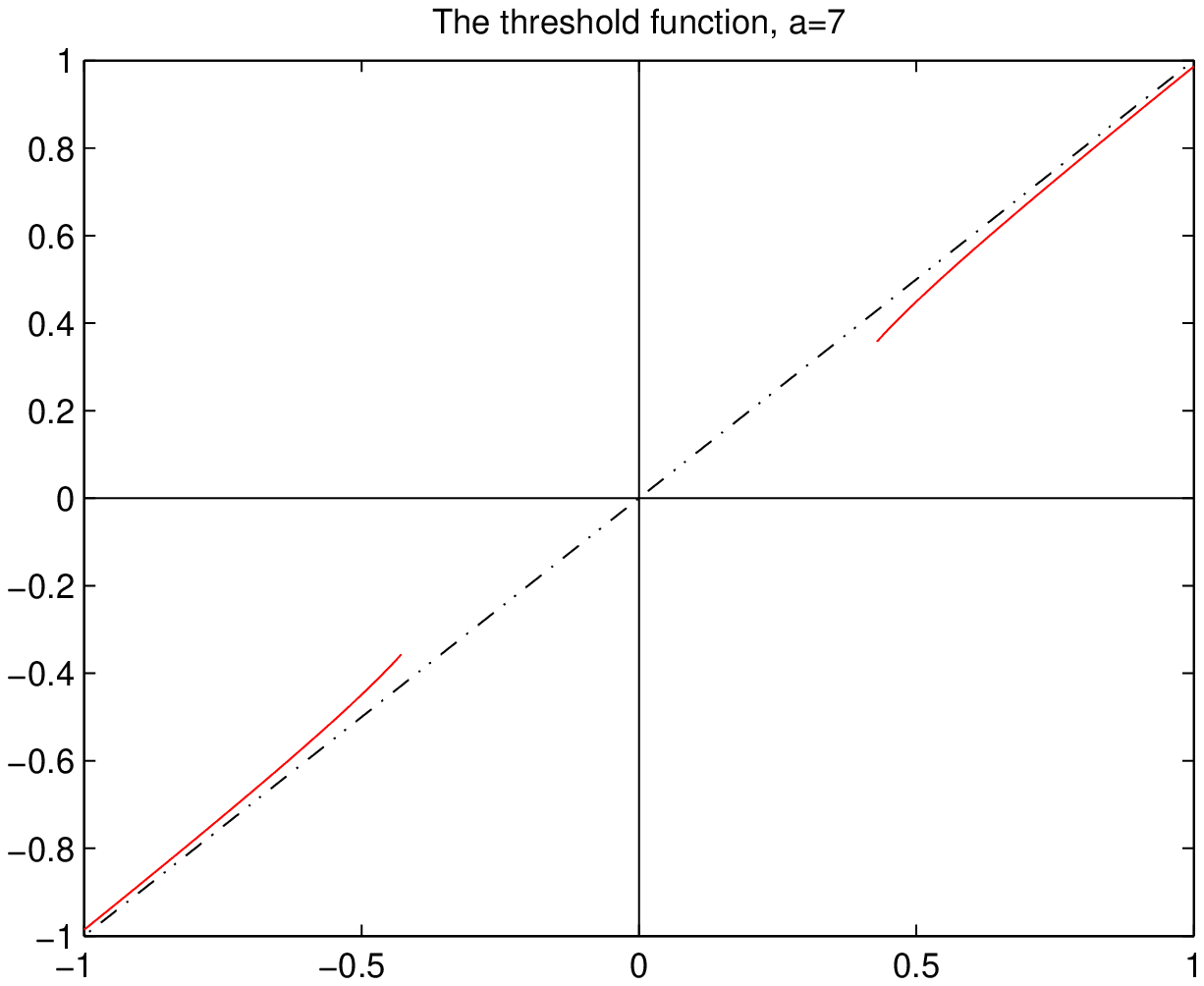}
  \end{minipage}
  \begin{minipage}[t]{0.49\linewidth}
  \centering
  \includegraphics[width=1.1\textwidth]{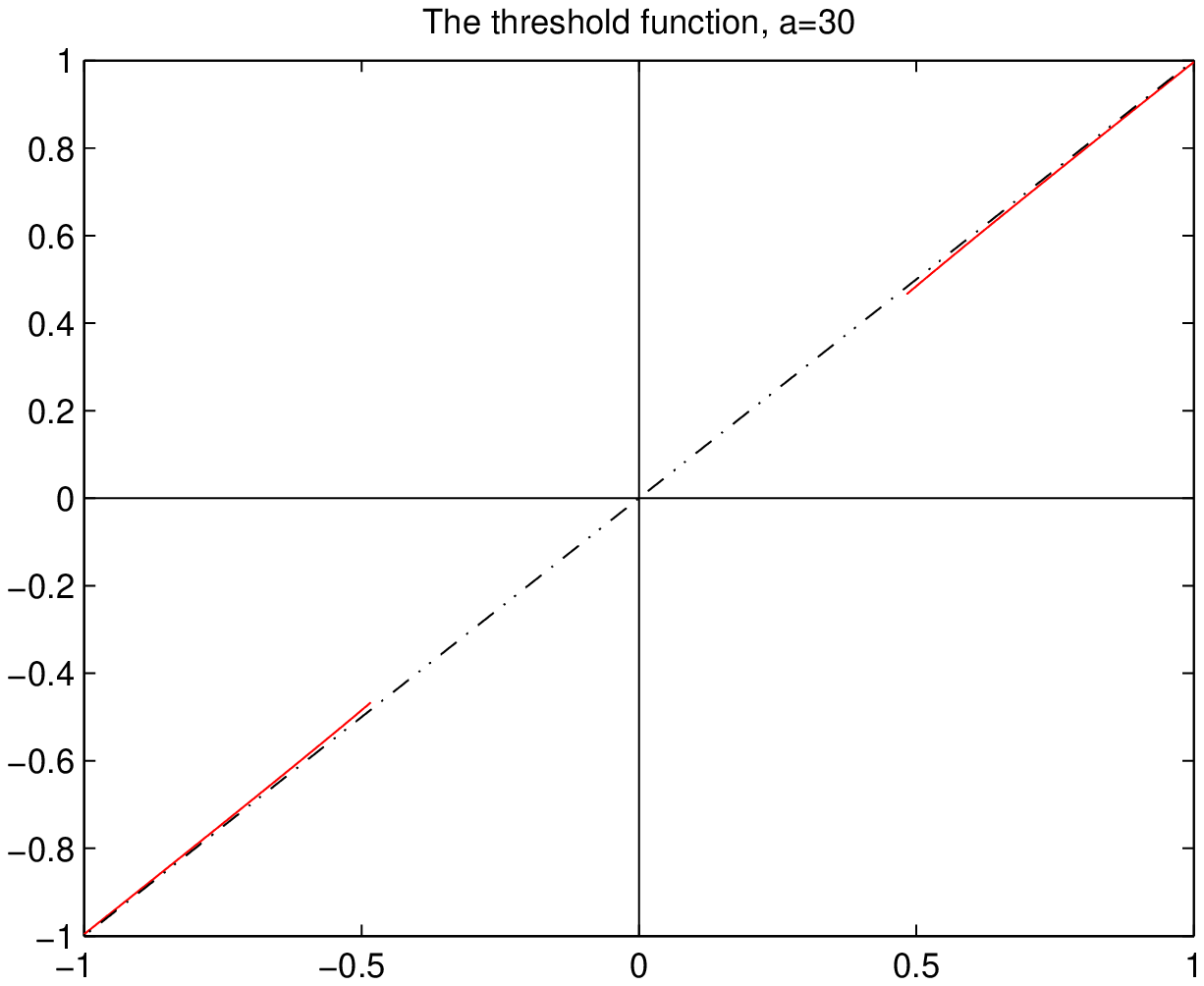}
  \end{minipage}
  \begin{minipage}[t]{0.49\linewidth}
  \centering
  \includegraphics[width=1.1\textwidth]{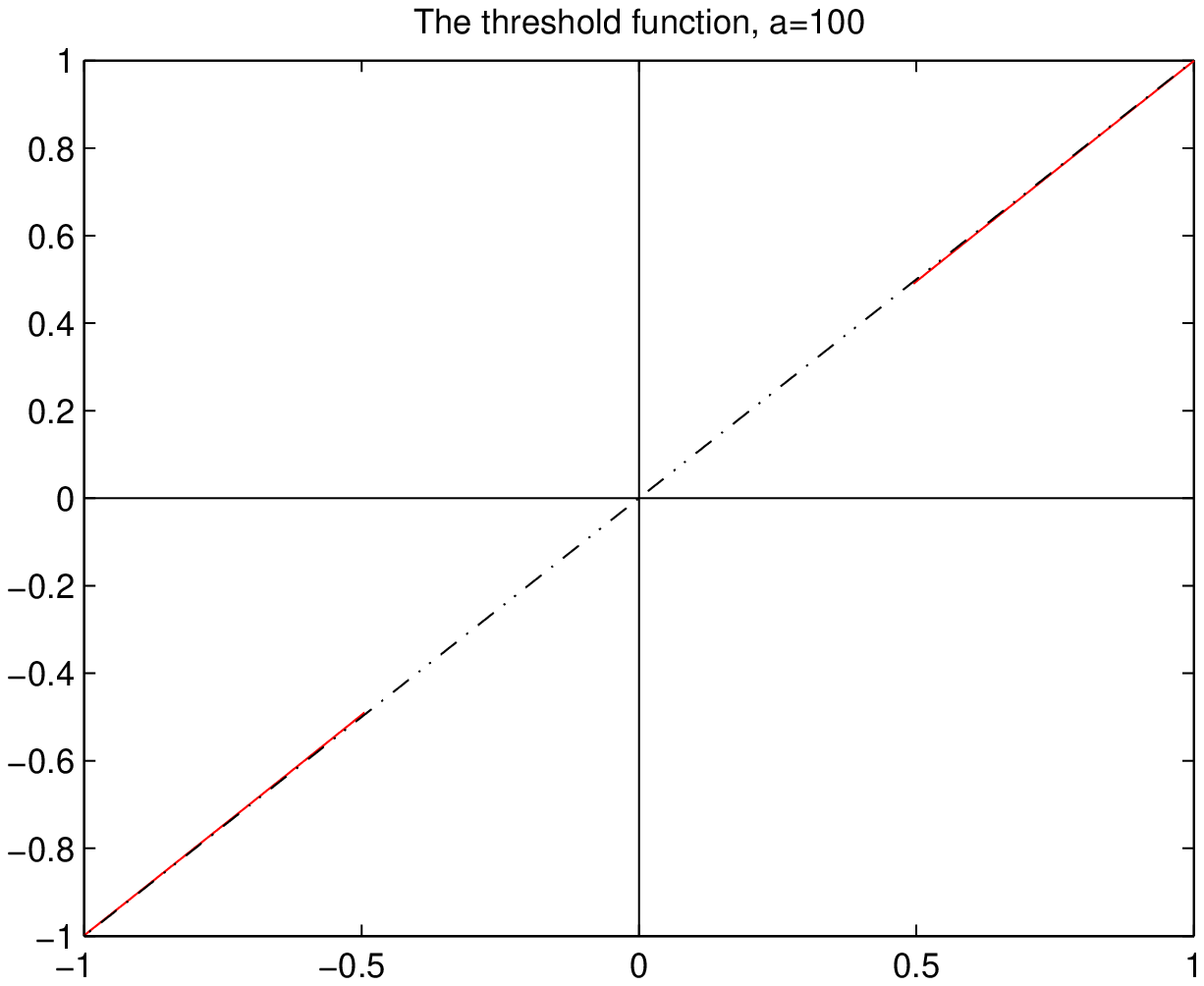}
  \end{minipage}
  \caption{The plots of $g_{a,\lambda}$ for a=1, 3, 5, 7, 30, 100, and $\lambda=0.25$.} \label{fig:2}
\end{figure}

\begin{definition}\label{de5}
Define the thresholding operator $\mathcal{T}_{a,\lambda}$ as a nonlinear analytically expressive operator, and can be specified by
\begin{equation}\label{equ24}
\mathcal{T}_{a,\lambda}(x)=(g_{a,\lambda}(x_{1}), \cdots, g_{a,\lambda}(x_{n}))^{T},
\end{equation}
where $g_{a,\lambda}(x_{i})$ is defined in Theorem \ref{th5}.
\end{definition}

The thresholding operator $\mathcal{T}_{a,\lambda}$ is a shrinking operator, and it is clear that if many of the absolute entries of vector $x$ are below the
threshold value $t^{\ast}$, the sparsity of $\mathcal{T}_{a,\lambda}(x)$ may be considerably lower than the sparsity of signal $x$ and leads to a sparse result.

\begin{definition}\label{de6}
Given any vector $v\in \Re^{n}$, define the projection map on $\Re^{n}_{+}$ by
$$\nabla_{+}(v):=\arg\min_{u\geq0}\{\|u-v\|_{2}^{2}\}=\max\{0, v\},$$
where the max operation is taken componentwise.
\end{definition}

\begin{theorem}\label{th6}
Let $v\in \Re^{n}$, we have
$$\mathcal{T}_{a,\lambda}(\nabla_{+}(v))=\arg\min_{x\geq0}\{\|x-v\|_{2}^{2}+\lambda P_{a}(x)\}$$
where $\mathcal{T}_{a,\lambda}$ and $\nabla_{+}$ are defined in Theorem \ref{th5} and Definition \ref{de6}.
\end{theorem}

\begin{proof}
Given any vector $v\in \Re^{n}$, let us introduce the following notations
$$x_{+}=x_{\mathcal{I}^{+}}\ \ \mathrm{and}\ \ x_{-}=x_{\mathcal{I}^{-}},$$
where
$$\mathcal{I}^{+}=\{i\ |\ \ i\in(1,2,\cdots,n), \ \ x_{i}\geq0\}$$
and
$$\mathcal{I}^{-}=\{i\ |\ \ i\in(1,2,\cdots,n), \ \ x_{i}<0\}.$$
Observe that the following relations hold
\begin{description}
  \item[(\romannumeral1)] $\|x\|_{2}^{2}=\|x_{+}\|_{2}^{2}+\|x_{-}\|_{2}^{2}$
  \item[(\romannumeral2)] $\|(x-v)_{+}\|_{2}^{2}+\|x_{-}\|_{2}^{2}=\|x-\nabla_{+}(v)\|_{2}^{2}$
  \item[(\romannumeral3)] $\|x_{-}\|_{2}^{2}=0\Leftrightarrow x_{i}=0\ \ \ \forall i\in \mathcal{I}^{-}$,
\end{description}
where the second relation follows from relation (i) and the fact that
$$(\nabla_{+}(v))_{i}=v_{i}$$
for any $i\in \mathcal{I}^{+}$ and
$$(\nabla_{+}(v))_{i}=0$$
for any $i\in \mathcal{I}^{-}$.

From the above facts (i)-(iii), we thus have that $\bar{x}\in \mathcal{T}_{a,\lambda}(\nabla_{+}(v))$ if and only if
\begin{eqnarray*}
\bar{x}&=&\arg\min_{x\geq0}\{\|x-v\|_{2}^{2}+\lambda P_{a}(x)\}\\
&=&\arg\min_{x\geq0}\{(\|(x-v)_{+}\|_{2}^{2}-\|(x-v)_{-}\|_{2}^{2})+\lambda P_{a}(x)\}\\
&=&\arg\min_{x\geq0}\{(\|(x-v)_{+}\|_{2}^{2}+\|x_{-}\|_{2}^{2}-2\sum_{i\in \mathcal{I}^{-}}x_{i}v_{i})\\
&&+\lambda P_{a}(x)\}\\
&=&\arg\min_{x\geq0}\{\|(x-v)_{+}\|_{2}^{2}+\lambda P_{a}(x): x_{i}=0\ \forall i\in \mathcal{I}^{-}\}\\
&=&\arg\min_{x\in \mathcal{R}^{n}}\{\|(x-v)_{+}\|_{2}^{2}+\lambda P_{a}(x): \|x_{-}\|_{2}^{2}=0\}\\
&=&\arg\min_{x\in \mathcal{R}^{n}}\{(\|(x-v)_{+}\|_{2}^{2}+\|x_{-}\|_{2}^{2})+\lambda P_{a}(x)\}\\
&=&\arg\min_{x\in \mathcal{R}^{n}}\{\|x-\nabla_{+}(v)\|_{2}^{2}+\lambda P_{a}(x)\}\\
&=&\mathcal{T}_{a,\lambda}(\nabla_{+}(v)).
\end{eqnarray*}
\end{proof}

Now, we show that the optimal solution to the problem ($FP_{a,\lambda}^{\geq}$) can be expressed as a thresholding operation.

For any fixed positive parameters $\lambda>0$, $\mu>0$, $a>0$ and $x,z\in \Re^{n}$, let
\begin{equation}\label{equ25}
C_{1}(x)=\|Ax-b\|_{2}^{2}+\lambda P_{a}(x)
\end{equation}
and its surrogate function
\begin{equation}\label{equ26}
\begin{array}{llll}
C_{2}(x,z)&=&\mu[C_{1}(x)-\|Ax-Az\|_{2}^{2}]+\|x-z\|_{2}^{2}
\end{array}
\end{equation}
where $\mu>0$ is a balancing parameter. Clearly, $C_{2}(x,x)=\mu C_{1}(x)$.

\begin{theorem}\label{th7}
For any fixed positive parameters $\lambda>0$, $\mu>0$ and $z\in \Re^{n}$, $\displaystyle\min_{x\geq0}C_{2}(x,z)$ equivalents to
\begin{equation}\label{equ27}
\min_{x\in \Re^{n}}\{\|x-\nabla_{+}(B_{\mu}(z))\|_{2}^{2}+\lambda\mu P_{a}(x)\}
\end{equation}
where $B_{\mu}(z)=z+\mu A^{T}(b-Az)$.
\end{theorem}

\begin{proof}
We first notice that, $C_{\mu}(x,z)$ can be rewritten as
\begin{eqnarray*}
C_{2}(x,z)&=&\|x-(z-\mu A^{T}Az+\mu A^{T}b)\|_{2}^{2}+\lambda\mu P_{a}(x)\\
&&+\mu\|b\|_{2}^{2}+\|z\|_{2}^{2}-\mu\|Az\|_{2}^{2}-\|z-\mu A^{T}A)z\\
&&+\mu A^{T}b\|_{2}^{2}\\
&=&\|x-B_{\mu}(z)\|_{2}^{2}+\lambda\mu P_{a}(x)+\mu\|b\|_{2}^{2}+\|z\|_{2}^{2}\\
&&-\mu\|Az\|_{2}^{2}-\|B_{\mu}(z)\|_{2}^{2}.
\end{eqnarray*}
Combined with Theorem \ref{th6}, we can get that $\displaystyle\min_{x\geq0}C_{2}(x,z)$, for any fixed $\mu,\ \lambda$ and $z\in \Re^{n}$,
equivalents to
$$\min_{x\in\Re^{n}}\{\|x-\nabla_{+}(B_{\mu}(z))\|_{2}^{2}+\lambda\mu P_{a}(x)\}.$$
\end{proof}

\begin{corollary}\label{co4}
Let $x^{\ast}=(x_{1}^{\ast},x_{2}^{\ast},\cdots, x_{n}^{\ast})^{T}$ be an optimal solution of $\displaystyle\min_{x\geq0}C_{2}(x,z)$
if and only if, for any $i, x_{i}^{\ast}$ solves the problem
$$\min_{x_{i}\in\Re}\{(x_{i}-(\nabla_{+}(B_{\mu}(z)))_{i})^{2}+\lambda\mu \rho_{a}(x_{i})\}.$$
\end{corollary}

\begin{theorem}\label{th8}
For any fixed $\lambda>0$ and $0<\mu<\frac{1}{\|A\|_{2}^{2}}$. If $x^{\ast}$ is an optimal solution of
$\displaystyle\min_{x\geq0}C_{1}(x)$, then $x^{\ast}$ is also an optimal solution of $\displaystyle\min_{x\geq0}C_{2}(x,x^{\ast})$, that is
$$C_{2}(x^{\ast},x^{\ast})\leq C_{2}(x,x^{\ast})$$
for any $x\geq0$.
\end{theorem}

\begin{proof}
Based on the definition of $C_{2}(x,z)$, we have
\begin{eqnarray*}
C_{2}(x,x^{\ast})&=&\mu[C_{1}(x)-\|Ax-Ax^{\ast}\|_{2}^{2}]+\|x-x^{\ast}\|_{2}^{2}\\
&=&\mu[\|Ax-b\|_{2}^{2}+\lambda P_{a}(x)]+\|x-x^{\ast}\|_{2}^{2}\\
&&-\mu\|Ax-Ax^{\ast}\|_{2}^{2}\\
&\geq&\mu[\|Ax-b\|_{2}^{2}+\lambda P_{a}(x)]\\
&=&\mu C_{1}(x)\\
&\geq&\mu C_{1}(x^{\ast})\\
&=&C_{2}(x^{\ast},x^{\ast}).
\end{eqnarray*}
\end{proof}

Theorem \ref{th8} shows that $x^{\ast}$ is an optimal solution to $\displaystyle\min_{x\in \Re^{n}}C_{\mu}(x,x^{\ast})$ as long as $x^{\ast}$
is an optimal solution of the problem ($FP_{a,\lambda}^{\geq}$). Combined with Theorem \ref{th7}, we derive the most important conclusion in this paper,
which underlies the algorithm to be proposed.

\begin{corollary}\label{co5}
Let $x^{\ast}$ be an optimal solution of the problem ($FP_{a,\lambda}^{\geq}$). Then $x^{\ast}$ is also an optimal solution of the following minimization problem
$$\min_{x\in \Re^{n}}\{\|x-B_{\mu}(x^{\ast})\|_{2}^{2}+\lambda\mu P_{a}(x)\}$$
for any $x\in \Re^{n}$.
\end{corollary}

Combining Corollary \ref{co5} and Theorem \ref{th5}, \ref{th6}, we can immediately conclude that the thresholding operation of the problem ($FP_{a,\lambda}^{\geq}$) can be given by
\begin{equation}\label{equ28}
x^{\ast}=\mathcal{T}_{a,\lambda\mu}(\nabla_{+}(B_{\mu}(x^{\ast})))
\end{equation}
where $\mathcal{T}_{a,\lambda\mu}$ is obtained by replacing $\lambda$ with $\lambda\mu$ in $\mathcal{T}_{a,\lambda}$.

With the thresholding representation (\ref{equ28}), the procedure of the NIT algorithm can be naturally defined as
\begin{equation}\label{equ29}
x^{k+1}=\mathcal{T}_{a,\lambda\mu}(\nabla_{+}(B_{\mu}(x^{k}))),
\end{equation}
where $B_{\mu}(x^{k})=x^{k}+\mu A^{T}(b-Ax^{k})$.

\subsection{Adjusting values for the regularization parameter}

In our algorithm, the cross-validation method is accepted to select the proper regularization parameter. Schwarz in \cite{schw55} demonstrated that when some prior information is known for a
regularization problem, this selection is more reasonably and intelligently.

We suppose that the vector $x^{\ast}$ of sparsity $r$ is the optimal solution of the problem $(FP_{a,\lambda}^{\geq})$,
without loss of generality, we set
\begin{equation}\label{equ30}
\begin{array}{llll}
&&(\nabla_{+}(B_{\mu}(x^{\ast})))_{1}\geq\cdots\geq(\nabla_{+}(B_{\mu}(x^{\ast})))_{r}\\
&&\geq(\nabla_{+}(B_{\mu}(x^{\ast})))_{r+1}=\cdots=0.
\end{array}
\end{equation}
By Theorem \ref{th5}, we have
$$(\nabla_{+}(B_{\mu}(x^{\ast})))_{i}>t^{\ast}\Leftrightarrow i\in \{1,2,\cdots,r\},$$
$$(\nabla_{+}(B_{\mu}(x^{\ast})))_{j}\leq t^{\ast}\Leftrightarrow j\in \{r+1, r+2,\cdots, n\},$$
where $t^{\ast}$ is the threshold value defined in (\ref{equ23}) obtained by replacing $\lambda$ with $\lambda\mu$.

According to $t_{3}^{\ast}\leq t_{2}^{\ast}$, we can get that
\begin{equation}\label{equ31}
\left\{
  \begin{array}{ll}
   (\nabla_{+}(B_{\mu}(x^{\ast})))_{r}\geq t^{\ast}\geq t_{3}^{\ast}=\sqrt{\lambda\mu}-\frac{1}{2a}; \\
   (\nabla_{+}(B_{\mu}(x^{\ast})))_{r+1}<t^{\ast}\leq t_{2}^{\ast}=\frac{\lambda\mu}{2}a,
  \end{array}
\right.
\end{equation}
which implies
\begin{equation}\label{equ32}
\frac{2(\nabla_{+}(B_{\mu}(x^{\ast})))_{r+1}}{a\mu}\leq\lambda\leq\frac{(2a(\nabla_{+}(B_{\mu}(x^{\ast})))_{r}+1)^{2}}{4a^{2}\mu}.
\end{equation}
The inequality (\ref{equ32}) helps us to set the strategy in selecting the best regularization parameter, and we denote $\lambda_{1}$ and $\lambda_{2}$
as the left and the right of above inequality respectively:
$$
\left\{
  \begin{array}{ll}
   \lambda_{1}=\frac{2(\nabla_{+}(B_{\mu}(x^{\ast})))_{r+1}}{a\mu}; \\
   \lambda_{2}=\frac{(2a(\nabla_{+}(B_{\mu}(x^{\ast})))_{r}+1)^{2}}{4a^{2}\mu}.
  \end{array}
\right.
$$
A choice of $\lambda$ is
\begin{equation}\label{equ33}
\lambda=\left\{
            \begin{array}{ll}
              \lambda_{1}, & {\mathrm{if}\ \lambda_{1}\leq\frac{1}{a^{2}\mu};} \\
              \lambda_{2}, & {\mathrm{if}\ \lambda_{1}>\frac{1}{a^{2}\mu}.}
            \end{array}
          \right.
\end{equation}

Since $x^{\ast}$ is unknown, and $x^{k}$ can be viewed as the best available approximation to $x^{\ast}$, a proper choice for the value of $\lambda$ at $k$-th iteration is given by
\begin{equation}\label{equ34}
\lambda=\left\{
            \begin{array}{ll}
             \lambda_{1}^{k}=\frac{2(\nabla_{+}(B_{\mu}(x^{k})))_{r+1}}{a\mu}, & {\mathrm{if}\ \lambda_{1}^{k}\leq\frac{1}{a^{2}\mu};} \\
             \lambda_{2}^{k}=\frac{(2a(\nabla_{+}(B_{\mu}(x^{k})))_{r}+1)^{2}}{4a^{2}\mu}, & {\mathrm{if}\ \lambda_{1}^{k}>\frac{1}{a^{2}\mu}.}
            \end{array}
          \right.
\end{equation}
That is, (\ref{equ34}) can be used to adjust the value of the regularization parameter $\lambda$ during iteration.

\begin{algorithm}[h!]
\caption{: NIT algorithm}
\label{alg:A}
\begin{algorithmic}
\STATE {Initialize: Choose $x^{0}$, $\mu_{0}=\frac{1-\varepsilon}{\|A\|_{2}^{2}}$ and $a$;}
\STATE {\textbf{while} not converged \textbf{do}}
\STATE \ \ \ {$z^{k}:=B_{\mu}(x^{k})=x^{n}+\mu A^{T}(y-Ax^{k})$;}
\STATE \ \ \ {$\lambda^{k}_{1}=\frac{2|B_{\mu}(x^{k})|_{r+1}}{a\mu}$; $\lambda^{k}_{2}=\frac{(2a|B_{\mu}(x^{k})|_{r}+1)^{2}}{4a^{2}\mu}$;}
\STATE \ \ \ {if\ $\lambda_{1}^{k}\leq\frac{1}{a^{2}\mu}$\ then}
\STATE \ \ \ \ \ \ \ {$\lambda=\lambda_{1}^{k}$; $t=\frac{\lambda\mu a}{2}$}
\STATE \ \ \ \ \ \ \ {for\ $i=1:\mathrm{length}(x)$}
\STATE \ \ \ {1.\ $(\nabla_{+}(B_{\mu}(x^{k})))_{i}>t$, $x^{k+1}_{i}=g_{\lambda\mu}((\nabla_{+}(B_{\mu}(x^{k})))_{i})$;}
\STATE \ \ \ {2.\ $(\nabla_{+}(B_{\mu}(x^{k})))_{i}\leq t$, $x^{k+1}_{i}=0$;}
\STATE \ \ \ {else}
\STATE \ \ \ \ \ \ \ {$\lambda=\lambda_{2}^{k}$; $t=\max\{\sqrt{\lambda\mu}-\frac{1}{2a}$, 0\}}
\STATE \ \ \ \ \ \ \ {for\ $i=1:\mathrm{length}(x)$}
\STATE \ \ \ {1.\ $(\nabla_{+}(B_{\mu}(x^{k})))_{i}>t$, $x^{k+1}_{i}=g_{\lambda\mu}((\nabla_{+}(B_{\mu}(x^{k})))_{i})$;}
\STATE \ \ \ {2.\ $(\nabla_{+}(B_{\mu}(x^{k})))_{i}\leq t$, $x^{k+1}_{i}=0$;}
\STATE \ \ \ {end}
\STATE \ \ \ {$k\rightarrow k+1$}
\STATE{\textbf{end while}}
\STATE{\textbf{return}: $x^{k+1}$}
\end{algorithmic}
\end{algorithm}

\section{The convergence of NIT algorithm} \label{section4}
In this section, we present the convergence of NIT algorithm under some specific conditions.
\begin{theorem} \label{th9}
Let $\{x^{k}\}$ be the sequence generated by iteration (\ref{equ29}) with the step size $\mu$ satisfying $0<\mu<\frac{1}{\|A\|_{2}^{2}}$. Then
\begin{description}
\item[$\mathrm{(1)}$] The sequence $\{C_{1}(x^{k})\}$ is decreasing;
\item[$\mathrm{(2)}$] $\{x^{k}\}$ is asymptotically regular, i.e., $\displaystyle\lim_{k\rightarrow\infty}\|x^{k+1}-x^{k}\|_{2}^{2}=0$;
\item[$\mathrm{(3)}$] Any accumulation point of $\{x^{k}\}$ is a stationary point of the problem $(FP_{a,\lambda}^{\geq})$.
\end{description}
\end{theorem}

\begin{proof}
(1) By the proof of Theorem \ref{th7}, we have
\begin{equation}\label{equ35}
C_{2}(x^{k+1}, x^{k})=\min_{x\geq0} C_{2}(x, x^{k}).
\end{equation}
Moreover, according to the definition of $C_{2}(x, z)$, we have
\begin{equation}\label{equ36}
\begin{array}{llll}
C_{1}(x^{k+1})&=&\frac{1}{\mu}[C_{2}(x^{k+1}, x^{k})-\|x^{k+1}-x^{k}\|_{2}^{2}]\\
&&+\|Ax^{k+1}-Ax^{k}\|_{2}^{2}.
\end{array}
\end{equation}
Since $0<\mu<\frac{1}{\|A\|_{2}^{2}}$, we can get that
\begin{equation}\label{equ37}
\begin{array}{llll}
C_{1}(x^{k+1})&=&\frac{1}{\mu}[C_{2}(x^{k+1}, x^{k})-\|x^{k+1}-x^{k}\|_{2}^{2}]\\
&&+\|Ax^{k+1}-Ax^{k}\|_{2}^{2}\\
&\leq&\frac{1}{\mu}[C_{2}(x^{k}, x^{k})-\|x^{k+1}-x^{k}\|_{2}^{2}]\\
&&+\|Ax^{k+1}-Ax^{k}\|_{2}^{2}\\
&=&C_{1}(x^{k})-\frac{1}{\mu}\|x^{k+1}-x^{k}\|_{2}^{2}\\
&&+\|Ax^{k+1}-Ax^{k}\|_{2}^{2}\\
&\leq&C_{1}(x^{k}).
\end{array}
\end{equation}
That is, the nonnegative sequence $\{x^{k}\}$ is a minimization sequence of function $C_{1}(x)$ for the constraint $x\geq0$, and
$$C_{1}(x^{k+1})\leq C_{1}(x^{k})$$
for all $k\geq0$.\\

(2) Let $\theta=1-\mu\|A\|_{2}^{2}$ and by the assumption about $\mu$, we have $\theta\in(0, 1)$, and
\begin{equation}\label{equ38}
\mu\|A(x^{k+1}-x^{k})\|_{2}^{2}\leq(1-\theta)\|x^{k+1}-x^{k}\|_{2}^{2}.
\end{equation}
By the inequality (\ref{equ37}), we can get that

\begin{equation}\label{equ39}
\begin{array}{llll}
&&\frac{1}{\mu}\|x^{k+1}-x^{k}\|_{2}^{2}-\|A(x^{k+1})-A(x^{k})\|_{2}^{2}\\
&&\leq C_{1}(x^{k})-C_{1}(x^{k+1}).
\end{array}
\end{equation}
Combing the inequalities (\ref{equ38}) and (\ref{equ39}), we have
\begin{eqnarray*}
\sum_{k=1}^{N}\|x^{k+1}-x^{k}\|_{2}^{2}&\leq&\frac{1}{\theta}\sum_{k=1}^{N}\|x^{k+1}-x^{k}\|_{2}^{2}\\
&&-\frac{\mu}{\theta}\sum_{k=1}^{N}\|Ax^{k+1}-Ax^{k}\|_{2}^{2}\\
&\leq&\frac{\mu}{\theta}\sum_{k=1}^{N}\{C_{1}(x^{k})-C_{1}(x^{k+1})\}\\
&=&\frac{\mu}{\theta}(C_{1}(x^{1})-C_{1}(x^{N+1}))\\
&\leq&\frac{\mu}{\theta}C_{1}(x^{1})
\end{eqnarray*}
where the last inequality holds by the fact that the sequence $\{C_{1}(x^{k})\}$ is decreasing.
Thus, the series $\sum_{k=1}^{\infty}\|x^{k+1}-x^{k}\|_{2}^{2}$ is convergent, which implies that
$$\|x^{k+1}-x^{k}\|_{2}^{2}\rightarrow 0 \ \ \mathrm{as}\ \ k\rightarrow\infty.$$

(3) Let $\{x^{k_{l}}\}$ be a convergent nonnegative subsequence of $\{x^{k}\}$, and denote $x^{\ast}$ as the limit point of $\{x^{k_{l}}\}$, i.e.,
\begin{equation}\label{equ40}
x^{k_{l}}\rightarrow x^{\ast}\ \ \mathrm{as}\ \ k_{l}\rightarrow \infty.
\end{equation}
Since
$$\|x^{k_{l}+1}-x^{\ast}\|_{2}\leq\|x^{k_{l}+1}-x^{k_{l}}\|_{2}+\|x^{k_{l}}-x^{\ast}\|_{2}$$
and
$$\|x^{k_{l}+1}-x^{k_{l}}\|_{2}+\|x^{k_{l}}-x^{\ast}\|_{2}\rightarrow 0 \ \ \mathrm{as}\ \ k_{l}\rightarrow \infty,$$
we have
\begin{equation}\label{equ41}
x^{k_{l}+1}\rightarrow x^{\ast}\ \ \mathrm{as}\ \ k_{l}\rightarrow \infty.
\end{equation}

Moreover, by iteration (\ref{equ29}), it follows that
$$x^{k_{l}+1}=\mathcal{T}_{a,\lambda\mu}(\nabla_{+}(B_{\mu}(x^{k_{l}}))),$$
and combined with Corollary \ref{co5}, we have
\begin{eqnarray*}
&&\|x^{k_{l}+1}-\nabla_{+}(B_{\mu}(x^{k_{l}}))\|_{2}^{2}+\lambda\mu P_{a}(x^{k_{l}+1})\\
&&\leq\|x-\nabla_{+}(B_{\mu}(x^{k_{l}}))\|_{2}^{2}+\lambda\mu P_{a}(x).
\end{eqnarray*}
Taking the limit of $X^{k_{l}+1}$ and using the continuity of $P_{a}$ as well as (\ref{equ40}) and (\ref{equ41}), we can immediately get that
\begin{eqnarray*}
&&\|x^{\ast}-\nabla_{+}(B_{\mu}(x^{\ast}))\|_{2}^{2}+\lambda\mu P_{a}(x^{\ast})\\
&&\leq\|x-\nabla_{+}(B_{\mu}(x^{\ast}))\|_{F}^{2}+\lambda\mu P_{a}(x).
\end{eqnarray*}
for any $x\geq0$, which implies that $x^{\ast}$ minimizes the following function
\begin{equation}\label{equ42}
\|x-\nabla_{+}(B_{\mu}(x^{\ast}))\|_{2}^{2}+\lambda\mu P_{a}(x),
\end{equation}
and we can conclude that
$$x^{\ast}=\mathcal{T}_{a,\lambda\mu}(\nabla_{+}(B_{\mu}(x^{\ast}))).$$
\end{proof}

\section{Numerical experiments}\label{section5}
In this section, we carry out a series of simulations to demonstrate the performance of NIT algorithm. To show the success rate of NIT algorithm in
recovering a signal with the different cardinality for a given measurement matrix $A$, we consider a random matrix $A$ of size $100\times256$ with
entries independently drawn by random from a Gaussian distribution of zero mean and unit variance, $N(0,1)$. By randomly generating some sufficiently
sparse nonnegative vectors $x_{0}$, we generate vectors $b$, and we know the sparsest solution to $Ax_{0} = b$. The stopping criterion is usually as following
$$\frac{\|x^{k+1}-x^{k}\|_{2}}{\|x^{k}\|_{2}}\leq \mathrm{Tol}$$
where $x^{k+1}$ and $x^{k}$ are numerical results from two continuous iterative steps and $\mathrm{Tol}$ is a given small number.

The success is measured by computing the relative $\ell_{2}$-error value
$$\mathrm{RE}=\frac{\|x^{\ast}-x_{0}\|_{2}}{\|x_{0}\|_{2}}$$
to indicate a perfect recovery of the original sparse nonnegative vector $x_{0}$. In our experiments, we set to $\mathrm{Tol}=1e-8$, and $\mathrm{RE}=1e-4$.
For each experiment, we repeatedly perform 100 tests and present average results.

\begin{figure}[h!]
 \centering
 \includegraphics[width=2.4in]{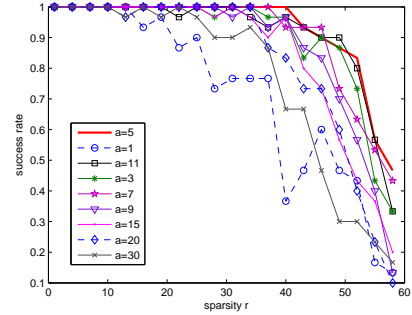}
\caption{The behaviors of the NIT algorithm for various values of $a>0$.}
\label{fig:2}
\end{figure}

\begin{figure}[h!]
 \centering
 \includegraphics[width=2.4in]{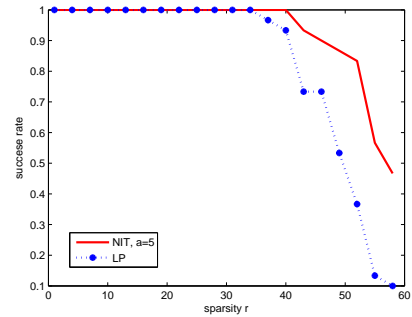}
\caption{The comparison of the NIT algorithm and linear programming (LP) in the recovery of sparse nonnegative signals.}
\label{fig:3}
\end{figure}

The graphs presented in Fig.\ref{fig:2} show the success rate of NIT algorithm in recovering the true (sparsest) solution with some different $a>0$, and $a=5$ seems to be the best strategy in our simulations. The graphs demonstrated in Fig.\ref{fig:3} show us that NIT algorithm can exactly recover the ideal signal until $r$ is around $40$, and linear programming (LP) is around $33$.

\begin{figure}[h!]
 \centering
 \includegraphics[width=2.4in]{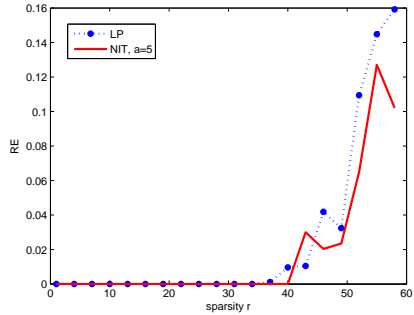}
\caption{The RE between the solution $x^{\ast}$ and the given signal $x_{0}$.}
\label{fig:4}
\end{figure}
From Fig.\ref{fig:4}, we can see that NIT algorithm always has the smallest relative $\ell_{2}$-error value, and as we can see, the NIT
algorithm ($a=5$) again has the best performance, with LP as the second.

\section{Conclusion}\label{section6}
In this paper, we replace the $\ell_{0}$-norm $\|x\|_{0}$ with a non-convex fraction function in the NP-hard problem $(P_{0}^{\geq})$, and translate this NP-hard problem into the problem $(FP_{a}^{\geq})$. We discussed the equivalence between $(FP_{a}^{\geq})$ and $(P_{0}^{\geq})$. Moreover, we also proved that the optimal solution of the problem $(FP_{a}^{\geq})$ could be approximately obtained by solving its regularization problem $(FP_{a,\lambda}^{\geq})$ for some proper smaller $\lambda>0$. The NIT algorithm is proposed to solve the regularization problem $(FP_{a,\lambda}^{\geq})$ for all $a>0$. Numerical experiments on sparse nonnegative
signal recovery problems show that our method performs effective in finding sparse nonnegative signals compared with the linear programming.


%

%
%

\section*{Acknowledgment}

The authors would like to thank editorial and referees for their comments which help us to enrich the content and improve the presentation of the results in this paper. The work was
supported by the National Natural Science Foundations of China (11771347, 11131006, 41390450, 11761003, 11271297) and the Science Foundations of Shaanxi Province of China (2016JQ1029, 2015JM1012).

\ifCLASSOPTIONcaptionsoff
  \newpage
\fi

\end{document}